\documentclass[11pt]{article}
\usepackage[english]{babel}
\usepackage{url}
\usepackage{amssymb,amsmath,verbatim,graphicx}

\newtheorem{theorem}{Theorem}[section]

\newtheorem{lemma}[theorem]{Lemma}
\newtheorem{cor}[theorem]{Corollary}

\newenvironment{proof}{\par\noindent\textbf{Proof}\hspace{1em}}{\qed}
\def\<{\langle}
\def\>{\rangle}

\newcommand{\cF}{\mathcal{F}}

\newcommand{\K}{\mathbb{K}}

\renewcommand{\L}{\mathbb{L}}

\newcommand{\cS}{\mathcal{S}}
\newcommand{\cB}{\mathcal{B}}

\newcommand{\cP}{\mathcal{P}}

\newcommand{\cL}{\mathcal{L}}

\newcommand{\HJ}{imbrex }
\newcommand{\N}{\mathbb{N}}

\def\EQ{symp}
\def\qed{{\hfill\hphantom{.}\nobreak\hfill$\Box$}}

\begin{document}

\author{Jeroen Schillewaert \thanks{Research supported by Marie Curie IEF grant GELATI (EC grant nr 328178)}\\ 
\small Department of Mathematics\\ [-0.8ex]
\small Imperial College\\ [-0.8ex]
\small London, England\\ 
\small\tt jschillewaert@gmail.com\\
\and 
Hendrik Van Maldeghem \\ 
\small Department of Mathematics\\[-0.8ex]
\small Ghent University\\[-0.8ex]
\small Ghent, Belgium\\
\small \tt hvm@cage.ugent.be
}

\title{Imbrex Geometries}
\date{}
\maketitle

\begin{abstract}
We introduce an axiom on strong parapolar spaces of diameter 2, which arises naturally in the framework of Hjelmslev geometries. This way, we characterize the Hjelmslev-Moufang plane and its relatives (line Grassmannians, certain half-spin geometries and Segre geometries). At the same time we provide a more general framework for a Lemma of Cohen, which is widely used to study parapolar spaces.
As an application, if the geometries are embedded in projective space, we provide a common characterization of (projections of) Segre varieties, line Grassmann varieties, half-spin varieties of low rank, and the exceptional variety $\mathcal{E}_{6,1}$ by means of a local condition on tangent spaces.
\end{abstract}

\section{Introduction}
Springer and Veldkamp \cite{Spr-Vel:68} introduced the \emph{Hjelmslev-Moufang planes} as geometries resembling Hjelmslev planes (because lines can meet in more than one point) and Moufang planes (because the coordinating structure is also an octonion algebra, though split, whereas the usual Moufang projective planes are defined over non-split ones). In the present paper we push the analogy a little bit further, using the more modern notion of \emph{parapolar spaces}. Indeed, although the Hjelmslev-Moufang planes are not Hjelmslev planes themselves,  we show that they satisfy a far more intuitive property of ``realistic geometry'', as Hjelmslev \cite{Hjel} himself was aiming at, and that property follows from the definition of Hjelmslev planes, but it does not characterize Hjelmslev planes. The basic observation made by Hjelmslev was that, if one draws lines ``close'' to each other (meaning that the sharp angle they define is very small), then it is hard to identify the intersection point, and it looks as if the lines have a little segment in common. Dually, if two points are very close to each other, then the joining line is hard to identify. Hjelmslev, and later Klingenberg \cite{Klingen}, included these observations in their geometries by introducing a neighbor relation which precisely indicates when two points and two lines are close to one another. However, in the approach of Hjelmslev and Klingenberg, the neighbor relation is an equivalence relation, and the equivalence classes define a ordinary projective or affine plane (the ``underlying'' plane). A consequence of this is the following intuitive property:

\begin{itemize}
\item[(Imb)]  \em Let $x$ be a point not neighboring two neighboring points $y_1,y_2$. Then there are unique lines $L_i$ joining $x$ and $y_i$, $i=1,2$, and $L_1$ is a neighbor of $L_2$.  
\end{itemize}

The proof is very simple: in the underlying plane, $y_1$ and $y_2$ define the same point, distinct from the point defined by $x$, hence the lines $L_1$ and $L_2$ must define the same lines and are consequently contained in the same equivalence class. 

We propose to take (Imb) as an axiom and combine this with the theory of parapolar spaces, where we want to view the symplecta as the lines of our geometry. In the approach of Hjelmslev, any two points are joined by at least one line; hence we consider strong parapolar spaces of diameter 2. Next, one must  define the notion of ``neighboring''. To do this, we refer back to the observation of Hjelmslev: if two points define a unique line, then these points are far enough from one another. Translated to the framework of parapolar spaces, two points at distance 2 define a unique symplecton. Consequently it is natural to define that two points are neighboring if they are collinear in the parapolar space.  For the dual, we take into account the residual nature of parapolar spaces associated to building geometries. More precisely, in the situations we are interested in, the residue in a point is again a parapolar space, and with the definition of neighboring points above, the neighbor relation on points is preserved under taking residues (i.e., if $y_1,y_2$ are two (non-)neighboring points collinear to a point $x$, then in the residue at $x$, the points defined by $y_1$ and $y_2$ are (non-)neighboring). So we also want the neighboring relation between symplecta to be preserved under taking residues. The only sensible way to define two symplecta to be neighbors then is when they intersect in a maximal singular subspace. This also implies that we should consider parapolar spaces of constant symplectic rank. %It is quite remarkable that this neighbor relation is not only preserved by taking residues, but also by the inverse procedure, see below.  

As a side remark, we mention that, in fact, the situation just described resembles in a certain sense better the reality than was the case with the Hjelmslev and Klingenberg planes. Indeed, in our case the neighbor relation is certainly not transitive, and this is more realistic: if one draws a number of points consecutively close to each other, then the first point may well be far from the last one. 

Now we translate the property (Imb) to the framework of parapolar spaces with the neighboring relation as just derived. So let $x$ be a point of a parapolar space of diameter 2, then $y_1,y_2$ are two collinear points (collinearity in the parapolar space), both at distance 2 from $x$. If there was a point on the line $y_1y_2$ collinear to $x$, then the symplecta through $x,y_1$ and $x,y_2$, which we denote by $\xi(x,y_1)$ and $\xi(x,y_2)$, respectively, would coincide, and hence be neighboring trivially. So we may assume that every point of the line $y_1y_2$ is at distance 2 from $x$. Then the symplecta $\xi(x,y_1)$ and $\xi(x,y_2)$ intersect in a maximal singular subspace of both of them. 

In conclusion, in this paper we will study strong parapolar spaces %of constant symplectic rank and 
with diameter 2 satisfying the following additional property, which we again refer to as (Imb), as an abbreviation of ``Imbrex''.

\begin{itemize}
\item[(Imb)]  \em Let $x$ be a point not collinear with any point of the line $L$. Let $y_1,y_2$ be distinct points on $L$. Then the symplecta $\xi(x,y_1)$ and $\xi(x,y_2)$ intersect in a singular subspace, which is maximal for both of them.   
\end{itemize}

We shall also need the more technical assumption that every set of mutually collinear points is contained in a maximal singular subspace. To ensure this we add the axiom that every sequence of nested singular subspaces is finite (it is the close analogue of the axiom for polar spaces ensuring finite rank, and all main examples satisfy it), and we call these parapolar spaces  {\em \HJ geometries}.  In fact, if the symplectic rank is at least 3, then one can classify imbrex geometries using a rather powerful theorem of Cohen \& Cooperstein \cite{Cohen,COCO} as updated by Shult \cite{Shu:12}. The merit of Property (Imb), however, lies in the fact that it allows to study the symplectic rank 2 case in a more general and conceptual way than was done before. Moreover, it is ready-made to generalize the characterization of Segre varieties  in \cite{Tha-Mal:13} to the other varieties in the second row of the split version of the extended Freudenthal-Tits Magic Square (FTMS), i.e., line Grassmannians of projective spaces and the variety associated to buildings of type $\mathsf{E_6}$, i.e., the variety associated to Springer and Veldkamp' Hjlemslev-Moufang planes (in contrast, the restricted FTMS just contains the Segre variety of two projective planes, and the line Grassmannian of projective $5$-space). 

The paper is organized as follows. After introducing the preliminaries in Section~\ref{pre}, we investigate in Section \ref{Imbrex} \HJ geometries of symplectic rank 2. We show that, if the symplecta are thick generalized quadrangles, then the maximal singular subspaces contain a lot of non-closing O'Nan configurations. This immediately implies a fundamental lemma of Cohen \cite{Cohen}, which was originally proved only for the case of classical generalized quadrangles, and later generalized by Shult and K. Thas \cite{KoenErnie} for all Moufang quadrangles (in fact, their proof shows that only the ``strong transitivity'' property of Moufang quadrangles is needed). In our setting, there is no restriction on the generalized quadrangles, and we also provide examples of \HJ geometries of symplectic rank 2 with thick symplecta (whose maximal singular subspaces are not isomorphic to projective spaces!). In this way, Cohen's lemma for thick symplecta is turned into a positive result, rather than merely showing nonexistence. Also, our result reveals the true geometric reason why the thick case in Cohen's setting cannot exist: it is not because the quadrangle has remarkable transitivity properties or enjoys the structure of pseudo-quadratic forms, but it is because in projective spaces all O'Nan configurations close (projective spaces are in fact characterized by that property).  

In Section~\ref{rank3}, we classify all \HJ geometries of symplectic rank at least 3. This will be an application of impressive work of Cohen and Cooperstein \cite{Cohen,COCO}.

In Section \ref{Application}, we apply our results to the theory of \emph{Mazzocca-Melone sets}.  Roughly, a Mazzocca-Melone set is a set of points in a projective space satisfying a far reaching generalization of the conditions stated by Mazzocca \& Melone \cite{Maz-Mel:84} that originally characterized finite quadric Veronesean varieties. In  \cite{JSHVMBIG} we provided a characterization of the Severi varieties over an arbitrary field. These precisely correspond to the varieties of the second row of the split version of the restricted FTMS using Mazzocca-Melone sets.  We here propose to alter one of the axioms from a global condition to a local one, inspired by the characterization of all  Segre %variety $\mathcal{S}_{2,2}$. The characterization of the second cell allows a generalization to cover all Segre 
varieties $\mathcal{S}_{m,n}(\K)$ (see \cite{Tha-Mal:13}). %This was achieved by a slight modification of the third axiom. Basically, one does not consider the global tangent space, but only a local part of it by restricting to hyperbolic quadrics meeting a certain fixed line (see below for more details). One can now extend the definition to ``arbitrary $d$'', thus defining in general
The thus defined \emph{local Mazzocca-Melone sets} will be classified in Section~\ref{Application}, and will entail a characterization of the varieties in the second row of the split version of the extended FTMS . %Based on our study of \HJ geometries we can classify all such sets. The motivation to study such sets is that the North-West $3\times3$ corner of the FTMS can be generalized to higher dimensions. Indeed, we basically obtain all Segre varieties and all line Grassmanniann varieties, plus, of course, the $26$-dimensional exceptional variety related to the standard module of groups of type $\mathsf{E}_6$ alluded to above.

\section{Preliminaries}\label{pre}

In this section we briefly introduce polar and parapolar spaces thus fixing notation. More information about parapolar spaces can be found in \cite{Shu:12}. %Readers familiar with this material can skip this section whereas novice readers might want to read more in \cite{Shu:12} for parapolar spaces or \cite{hvm} for generalized quadrangles. 

\subsection{Polar spaces}

Polar spaces are essentially the geometries associated with pseudo-quadratic forms, except that in the rank 2 case no classification is feasible since there exist free constructions. Polar spaces have been introduced by Veldkamp \cite{Veldkamp:59}, later on included in the theory of buildings by Tits \cite{Tits:74}, and around the same time the axioms have been simplified by Buekenhout \& Shult \cite{Bue-Shu:74}. It is the latter point of view we take here. 

Let $\Gamma=(\cP,\cL,*)$ be a point-line geometry ($\cP$ is the set of points, $\cL$ the set of lines, and $*$ a symmetric incidence relation). We will not consider geometries with repeated lines, so from now on we view $\cL$ as a subset of the power set of $\cP$, and $*$ is inclusion made symmetric. The \emph{incidence graph} is the bipartite graph on $\cP\cup\cL$ with $*$ as adjacency relation. The dual of $\Gamma$ is the point-line geometry $(\cL,\cP,*)$. A \emph{subspace} of $\Gamma$ is a subset $S$ of the point set such that, if two points $a,b$ belong to $S$, then all lines containing both $a$ and $b$ are contained in $S$. Points contained in a common line will be called \emph{collinear}, dually, lines sharing at least one point are called \emph{concurrent}. A \emph{singular} subspace is a subspace every two points of which are collinear.  Note that the empty set and a single point are legible singular subspaces. Now, $\Gamma$ is called a \emph{polar space of rank $r$} if the following conditions hold. 

\begin{itemize}\addtolength{\itemsep}{-0.4\baselineskip}
\item[(PS1)] Every line contains at least $3$ points.
\item[(PS2)] No point is collinear with all other points.
\item[(PS3)] Every nested sequence of singular subspaces has at most length $r+1$ and there exists such sequence of length $r+1$.
\item[(PS4)] For any point $x$ and any line $L$, either one or all points on $L$ are collinear with $x$. 
\end{itemize}

A \emph{generalized quadrangle} is a polar space of rank 2, the dual of a polar space of rank 2, or the point-line geometry defined by a $(2\times N)$-grid (with $N$ any cardinal number at least 2), or the dual of the latter. A generalized quadrangle is \emph{thick} if it is both a polar space and a dual polar space, equivalently, if every element is incident with at least 3 other elements. If $T\subseteq\cL$, then we denote by $T^\perp$ the set of lines concurrent with every member of $T$. A pair of non-concurrent lines $\{L,M\}$ of a generalized quadrangle is called \emph{regular} if $(\{L,M\}^\perp)^\perp=\{L',M'\}^\perp$, for some distinct $L',M'\in\{L,M\}^\perp$. If for a line $L$ every pair $\{L,M\}$ with $M$ non concurrent with $L$ is regular, then the line $L$ is called \emph{regular}.  A \emph{subquadrangle} of the generalized quadrangle $\Gamma=(\cP,\cL,*)$ is a subset of $\cP$, together with a subset of $\cL$ on which $*$ induces a generalized quadrangle. A subquadrangle $\Gamma'$ is \emph{ideal} if for every point $x$ in $\Gamma'$, every line of $\Gamma$ incident with $x$ also belongs to $\Gamma'$.

We note that polar spaces are partial linear spaces, i.e., two collinear points $x,y$ determine exactly one line, which we sometimes denote by $xy$.

\subsection{Parapolar spaces}

Parapolar spaces were introduced to capture the spherical buildings of exceptional type (\emph{spherical buildings} comprise projective spaces, polar spaces and the ones of exceptional type). Since we only need strong parapolar spaces of diameter 2, we only introduce these. 

Let again 
$\Gamma=(\mathcal{P},\mathcal{L},*)$ be a point-line geometry. A subspace $S$ is called \emph{convex} if for any pair of points $\{x,y\}\subseteq S$, every point contained in or incident with a line of any shortest path between $x$ and $y$ (in the incidence graph) is contained in $S$. Also, $\Gamma$ is called \emph{connected} if its incidence graph is connected.  Now, $\Gamma$ is called a {\em strong parapolar space of diameter $2$} if the following two assertions hold:
\begin{itemize}\addtolength{\itemsep}{-0.4\baselineskip}
\item[(PPS1)] $\mathcal{P},\mathcal{L}$ is a connected space such that for every point $x$ and every line $L$  either 0, or exactly one or all points of $L$ are collinear with $x$ and all possibilities occur.
\item[(PPS2)] For every pair of distinct non-collinear points $x$ and $y$ in $\mathcal{P}$, the smallest convex subspace containing $x$ and $y$ is a polar space of rank at least 2.
\end{itemize}

The convex subspaces of (PPS2) are called the \emph{symplecta}, or briefly, the \emph{symps}, of $\Gamma$. By Axiom (PPS2), any two non-collinear points $x,y\in\cP$ are contained in a unique symp and we denote it by $\xi(x,y)$.  

Usually, the third axiom of a parapolar space reads
\begin{itemize}
\item[(PPS3)] Every line is contained in at least one symplecton,
\end{itemize}

but this axiom is automatically satisfied in our case (we restricted to \emph{strong} and \emph{diameter}~2, see\cite{JSHVMBIG}). 

%Here a set of points $R$ is called a {\em subspace} if for every line $L\in\mathcal{L}$ either $|L\cap R|\leq 1$ or $L\subset R$. Moreover {\em convexity} of $[x,y]$ means that every point on a path of a length 2 between $x$ and $y$ belongs to $[x,y]$.
%A subspace $R$ is called {\em singular} if every two distinct points of $R$ are collinear. We will henceforth often use symp to abbreviate symplecton. 

It follows from  Corollary 13.3.3 of \cite{Shu:12} that any symplecton is uniquely determined by any pair of noncollinear points contained in it. We will use this without further notice.
 From (PPS2) we immediately obtain the following two lemmas.
\begin{lemma}[The Quadrangle Lemma] \label{TQL}
Let $L_1,L_2,L_3,L_4$ be four (not necessarily pairwise distinct) singular lines such that $L_i$ and $L_{i+1}$ (where $L_{5}=L_{1}$) share a (not necessarily unique) point $p_{i}$, $i=1,2,3,4$, and suppose that $p_1$ and $p_3$ are not collinear. Then $L_1,L_2,L_3,L_4$ are contained in a unique common \EQ.
\end{lemma}

\begin{lemma}\label{subspace}
Let $p\in X$ and let $H$ be a symp not containing $p$. Then the set of points of $H$ collinear with $p$   constitutes a singular subspace of $H$.
\end{lemma}

Moreover we will use the following 
\begin{lemma}\label{proper}
Let $x,y\in X$ be collinear. Then there is a \EQ\ containing $x$ and not containing~$y$. 
\end{lemma}

\begin{proof}
Since $\Gamma$ contains at least two symps by (PPS1), it is easy to see, using the previous lemmas, that there are at least two different \EQ s $H,H'$ containing $x$. 
Suppose both contain $y$. Select a point $p\in H\setminus H'$ collinear to $x$ but not collinear to $y$, and a point $p'\in H'\setminus H$ collinear to $x$ but not collinear to $p$ (these points exist by Lemma~\ref{subspace} and since $H\cap H'$ is a singular subspace). By convexity, $\xi(p,p')$ contains $x$, but it does not contain $y$ because the intersection $H\cap H''$ would otherwise not be a singular subspace, as it would contain the non-collinear points $p$ and $y$.
\end{proof}

\subsection{Imbrex geometries}

We will require an additional axiom, which basically says that maximal singular subspaces exist.

\begin{itemize}
\item[(PPS4)] Every nested sequence of singular subspaces has finite length.
\end{itemize}

We call a strong parapolar space of rank 2 an \emph{imbrex geometry} if it additionally satisfies (PPS4) and the axiom (Imb) stated in the introduction. %We call it \emph{proper} if there are at least two distinct symps. 
We now show that all symps have the same rank. This common rank will then be called the \emph{symplectic rank} of the imbrex geometry.

\begin{lemma}
In an imbrex geometry, all symps have the same rank.
\end{lemma}

\begin{proof}
Let $H,H'$ be two symps. We first note that (PPS2) implies that the graph of all symps, adjacent when they meet in a nonempty subspace, is connected. Hence we may assume that $H\cap H'$ contains some point $x$. Let $y\in H$ and $y'\in H'$ be two points, both not collinear to $x$, and hence $y'\notin H$, $y\notin H'$. If $y$ and $y'$ are collinear, then suppose that the line $yy'$ contains a point $z$ collinear to $x$. It follows that the symp $\xi(x,y)=H$ contains $z$ and hence $y'$, a contradiction. Hence by (Imb), the symps $\xi(x,y)=H$ and $\xi(x,y')=H'$ have the same rank. So we may assume that $y$ and $y'$ are not collinear.  The symp $\xi(y,y')$ shares a singular subspace $S$ with $H$ and a singular subspace $S'$ with $H'$. We distinguish two possibilities.
\begin{enumerate}
\item Suppose the rank of $\xi(y,y')$ is at least $3$, or $\xi(y,y')$ is a thick generalized quadrangle. Then it is easy to see that we can find a point $z$ not in $S\cup S'$, and so not in $H\cup H'$, but collinear to both $y,y'$. As before, the lines $yz$ and $y'z$ do not contain any point collinear to $x$. Hence the symps $\xi(x,y)=H$ and $\xi(x,z)$ have the same rank, as well as $\xi(x,y')=H'$ and $\xi(x,z)$.
\item Suppose $\xi(y,y')$ is a non-thick generalized quadrangle, i.e., a grid.   If we cannot find a point $z$ as in the previous case, then the only possibility is that $S$ and $S'$ are disjoint lines of $\xi(y,y')$. Since both $S$ and $S'$ contain at least three points, we can find collinear points $u\in S$ and $u'\in S'$ such that $x$ is not collinear to either $u$ or $u'$. By (Imb), the symps $\xi(x,u)=H$ and $\xi(x,u')=H'$ intersect in a maximal singular subspace of both, implying $H$ and $H'$ have the same rank. 
\end{enumerate}     
\end{proof}

%\subsection{Generalized quadrangles}

%A {\em generalized quadrangle} is a point-line incidence geometry $\mathcal{G}=(\mathcal{P},\mathcal{L},\mathrm{I})$ such that
%\begin{itemize}\addtolength{\itemsep}{-0.4\baselineskip}
%\item There exists an $s\geq 1$ such that on every line there are exactly $s + 1$ points. There is at most one point on two distinct lines.
%\item There exists a $t\geq 1$ such that through every point there are exactly $t + 1$ lines. There is at most one line through two distinct points.
%\item For every point $p$ not on a line $L$, there is a unique line $M$ and a unique point $q$, such that $p$ is on $M$, and $q$ on $M$ and $L$.
%\end{itemize}
%The pair $(s,t)$ form the parameters of the generalized quadrangle. These are allowed to be infinite. If either $s$ or $t$ is 1 we call the generalized quadrangle {\em thin}. A quadrangle of order $(s,1)$ is called a {\em grid}, whereas one of order $(1,t)$ is a {\em dual grid}. By switching the roles of points and lines in a generalized quadrangle $\mathcal{G}$ of order $(s,t)$, we obtain the {\em dual generalized quadrangle} $\mathcal{G}^{D}$.

\section{Imbrex geometries of symplectic rank 2} \label{Imbrex}
%\subsection{\HJ geometries of symplectic rank two}\label{rank2}
Throughout this section, we let $\Gamma=(\cP,\cL,*)$ be an imbrex geometry of symplectic rank 2. Our main aim is to show that every maximal singular subspace of $\Gamma$ contains at least one \emph{non-closed O'Nan configuration}, i.e., four distinct lines pairwise intersecting in a point, except for one pair, which is disjoint. 

\begin{lemma}\label{pointcol}
Let $x,q_{1},q_{2}\in\cP$ with $q_1$ collinear to $q_2$. Suppose no point of the line $q_1q_2$ is collinear to 
$x$. Then some point of the line $L:=\xi(x,q_{1})\cap \xi(x,q_{1})$ is collinear with all points on $q_1q_2$. 
\end{lemma}
\begin{proof}
Note that $L$ is indeed a line by (Imb). Now, in the generalized quadrangle $\xi(x,y_i)$ there is a unique point $z_i$ on $L$ collinear to $y_i$, $i=1,2$. If $z_1\neq z_2$, then Lemma \ref{TQL} yields $\xi(x,q_{1})=\xi(x,q_{2})$, implying that $x$ is collinear with a point of $L$, a contradiction. Hence $z_1=z_2\in L$ is collinear to both $y_1$ and $y_2$. Axiom (PPS1) completes the proof.
\end{proof}

We define a new point-line geometry $\Delta=(\mathcal{P},\mathcal{B},\star)$,  where we call the elements of $\cB$ \emph{blocks} to avoid confusion with the lines of $\Gamma$, where $\cB$ is the family of  maximal singular subspaces of $\Gamma$, and where $\star$ is containment made symmetric. We will need the following auxiliary results.

\begin{lemma}\label{far}
If $\Gamma$ contains at least two distinct symps, then it enjoys the following two properties.
\begin{itemize}\addtolength{\itemsep}{-0.4\baselineskip}
\item[$(i)$] For each point $r$, there exists a line $L_{r}$ no point of which is collinear to $r$.
\item[$(ii)$] For each line $M$, there exist a point $y_{M}$ not collinear with any point of $M$. 
\end{itemize}
\end{lemma}

\begin{proof}
For $(i)$ consider a symp $H$ through $r$, a point $y$ not in $H$ and a point $z$ in $H$ not collinear with either $r$ or $y$ (it is easy to check that this is always possible). Consider two different lines $L_{1},L_{2}$ in $\xi(y,z)$ through $y$. If $r$ is collinear with points $p_{1},p_{2}$ on $L_{1},L_{2}$, respectively, then $r\in \xi(p_1,p_2)=\xi(y,z)$, implying $H=\xi(r,z)=\xi(y,z)$ contains $y$, a contradiction. Hence at least one of $L_{1}$
or $L_{2}$ will do.

For $(ii)$, we argue by contradiction. Consider a point $p$ on $M$. Consider a symp $H$ through $p$ and a point
$x\in H$ not collinear to $p$.  
By assumption there exists a point on $M$ collinear with $x$. Hence
$M\subseteq\xi(x,p)=H$. As $H$ was arbitrary this contradicts Lemma \ref{proper}.
\end{proof}

 Lemma \ref{pointcol}  and Lemma \ref{far}$(ii)$ imply the following result.
\begin{cor}\label{lineblock}
Every line of $\Gamma$ is properly contained in a block of $\Delta$.
\end{cor}

\begin{proof}
Let $M\in\cL$ and let $p_M$ be a point not collinear to any point of $M$ (cf.~Lemma~\ref{far}$(ii)$). Then Lemma~\ref{pointcol} implies that there is a point $z$ collinear to all points of $M$. Let $S_0=M\cup\{z\}$. Define inductively $S_{i+1}$ as $S_i$ union all the lines containing at least two point of $S_{i}$. Let $S$ be the infinite union of all $S_i$, $i\in\mathbb{N}\cup\{0\}$. We claim that $S$ is a singular subspace. Indeed, by construction it is a subspace. Clearly, the claim follows if we show that all points of each $S_i$, $i\in\N\cup\{0\}$, are collinear to one another.   Since this is clearly the case for $S_0$, we may argue by induction and assume that all points of $S_{i-1}$ are collinear, $i\in\N$. Let $x,y\in S_i$. We may assume $y\notin S_{i-1}$. If $x\in S_{i-1}$, then let $z_1,z_2\in S_{i-1}$ with $y\in z_1z_2$. Since $x$ is collinear with both $z_1$ and $z_2$, Axiom (PPS1) implies that $x$ is collinear to $y$. If $x\notin S_{i-1}$, then let $u_1,u_2\in S_{i-1}$ be such that $x\in u_1u_2$. By the previous argument both $u_1$ and $u_2$ are collinear to $y$, hence again by (PPS1) the assertion follows.

So $S$ is a singular subspace properly containing $M$. More generally, the above proof can be used to show that every set of mutually collinear points is contained in a smallest singular subspace. Now let $\cF$ be the family of singular subspaces of $\Gamma$ containing $S$. We claim that every pair of elements of $\cF$ has a least upper bound (join) in the poset defined by $\cF$. Indeed, let $S_1,S_2\in\cF$ and suppose that some point $y_1$ of $S_1$ is not collinear to some point $y_2$ of $S_2$. We pick two points $x_1,x_2\in S$. Then Lemma~\ref{TQL} yields a symp containing $x_1,x_2,y_1,y_2$, contradicting the fact that the symplectic rank is 2. 

Now let $S^\infty$ be the last member of a maximal nested sequence contained in $\cF$ (which is finite and exists by (PPS4)). Then every member of $\cF$ is contained in $S^\infty$ by the previous paragraph and the maximality of the sequence. The corollary is proved.
\end{proof}

\begin{lemma}\label{moreprop}
\begin{itemize}\addtolength{\itemsep}{-0.4\baselineskip}
\item[$(i)$] Two blocks intersect at most in one point.
\item[$(ii)$] A point $p$ not contained in a block $B$ is collinear with at least one point $r\in B$ 
\end{itemize}
\end{lemma}
\begin{proof}
Suppose two blocks $B_1,B_2$ share at least two points $x_1,x_2$. Let $y_1\in B_1\setminus B_2$ and $y_2\in B_2\setminus B_1$ with $y_1$ not collinear to $y_2$ (this is possible since both $B_1$ and $B_2$ are maximal singular subspaces).  Lemma~\ref{TQL} yields a symp containing $x_1,x_2,y_1,y_2$, contradicting the fact that the symplectic rank is 2.  This shows $(i)$.

For $(ii)$ consider a line $L\subset B$. If $p$ is collinear with a point of $L$ we are done.
Otherwise, by Lemma \ref{pointcol} and the proof of Corollary~\ref{lineblock}, there is a block $B'$ containing $L$ such that $p$ is collinear with a point of $B'$. But by
$(i)$, $B=B'$, establishing the proof.
\end{proof}

%A subquadrangle $\mathcal{Q'}$ of a generalized quadrangle $\mathcal{Q}$ is called {\em ideal} if every pencil of $Q'$ coincides with the corresponding pencil of $Q$. 
 
\begin{lemma}
\begin{itemize}\addtolength{\itemsep}{-0.4\baselineskip}
\item[$(i)$] The point-line geometry $\Delta$ is a generalized quadrangle where each line contains at least three points.
\item[$(ii)$] Every symp is an ideal subquadrangle of $\Delta$. In particular, the symplecta are either all thick or all non-thick. 
\end{itemize}
\end{lemma}

\begin{proof}
For $(i)$, we check the axioms of a polar space of rank 2.  Corollary \ref{lineblock} implies that every point is contained in at least two blocks and clearly every block contains at least three points. Whence (PS1). Since by Corollary~\ref{lineblock}, collinearity in $\Delta$ coincides with collinearity in $\Gamma$, (PS2) follows. Maximality of the blocks implies (PS3) with rank equal to 2.  

As for (PS4), consider a point $x$ not contained in a block $B$. 
Consider a line $L\subset B$ and suppose $x$ is not collinear with any point on $B$. Then by Lemma \ref{pointcol} there exists a block $B'$ containing $L$ and such that there is a point $y\in B'$ collinear with $x$. But, as before, $B=B'$ and so there is at least one point $y$ of $B$ collinear to $x$.
 If $x$ were collinear to at least two points of $B$, then we would find a maximal subspace intersecting $B$ in at least two points, a contradiction as before. 

For $(ii)$, consider a point $x$ and a symp $H$ through $x$. We need to show that each block $B$ containing $x$ intersects $H$ in a line. Consider $y\in H$ not collinear with $x$. By Lemma \ref{moreprop}$(ii)$, $y$ is collinear with a point $z\in B$ and hence it follows that $z\in\xi(x,y)= H$, establishing the proof. 
 \end{proof}

Let $B$ be a block of $\Delta$ and let $H\not\supseteq B$ be a symp of $\Gamma$, viewed as ideal subquadrangle of $\Delta$. Since $H$ is ideal, $B$ does not contain any point of $H$. Let $u\in H$ be arbitrary. Then there is a unique block $B_u$ through $u$ intersecting $B$ in a point. The block $B_u$ intersects $H$ in a line $L_u$. For $v\in H$, we easily have $L_u\cap L_v=\emptyset$ or $L_u=L_v$. Hence we obtain a partition of $H$ into lines, which is usually called a \emph{spread} of $H$. Here, the spread is said to be \emph{induced (by $B$)}. These induced spreads have an interesting property.

\begin{lemma}\label{pairregular}
\begin{itemize}\addtolength{\itemsep}{-0.4\baselineskip}
\item[$(i)$] Every pair of non-concurrent lines of any symp of $\Gamma$ is regular. 
\item[$(ii)$] If $L,M$ are two distinct lines of some induced spread $\cS$ in some symp $H$, then every member of $(\{L,M\}^\perp)^\perp$ belongs to the spread. Moreover, if the spread is induced by the block $B$, then the point-line geometry $\beta$ induced on $B$ by the lines of $\Gamma$ contained in $\cS$ contains a subspace which is isomorphic to the point-line geometry $\sigma$ induced on $\cS$ by the ``double perps'', i.e., the point-line geometry with point set $\cS$ and lines the sets $(\{L,M\}^\perp)^\perp$, with $L,M\in\cS$.
\end{itemize}
\end{lemma}
\begin{proof}
Consider a symp $H$ and and let $L_1,L_2$ be two non-concurrent lines of $H$. Let $B_1,B_2$ be the blocks containing $L_1,L_2$, respectively. Let $x_1\in B_1\setminus L_1$. Then there is a unique point $x_2\in B_2$ collinear with $x_1$. Since $H$ is a subquadrangle of $\Delta$, the point $x_2$ does not belong to $L_2$.  We note that the block $B$ containing $x_1$ and $x_2$ does not contain a point of $H$ because if $u\in B\cap H$, then the unique point on $B_1$ collinear with $u$ would be contained in $L_1$ since $H$ is a subquadrangle. But that point is $x_1$, a contradiction. 

Let $M_1,M_2\in\{L_1,L_2\}^\perp$ be distinct. Put $y_{i,j}= L_i\cap M_j$. The symp $\xi(x_1,y_{2,j})$, $j=1,2$, contains the lines $x_1x_2$ and $M_j$. Hence, for any point $x\in x_1x_2$, there is a unique point $y_j\in M_j$ collinear to $x$. If $y_1$ were not collinear to $y_2$, then $H=\xi(y_1,y_2)$ would contain $x$, a contradiction.  Hence $y_1$ and $y_2$ are collinear. Varying $M_2$ over $\{L_1,L_2\}^\perp\setminus\{M_1\}$, we see that the line $y_1y_2$ belongs to $(\{L_1,L_2\}^\perp)^\perp$. This shows $(\{L_1,L_2\}^\perp)^\perp=\{M_1,M_2\}^\perp$, and hence the pair $\{L_1,L_2\}$ is regular. 

But, as is obvious from the previous paragraph, all elements of $\{M_1,M_2\}^\perp$ belong to the spread induced by $B$, and the corresponding blocks intersect $B$ precisely in (all) points of the line $x_1x_2$. Hence the mapping $y_1y_2\mapsto x$ defines an injective morphism from $\sigma$ to $\beta$.  This completes the proof of the lemma.
%
% a line $M$ which have nothing in common. Then we obtain a spread $S$ in $Q$ by looking at all blocks through a point of $B$. Consider two lines $L_{1},L_{2}$ of $S$ and a line $N_{1}$ intersecting both of them. Then by Lemma \ref{TQL} we have that $L_{1}$ and $M$ are grid-bijective. Consider another line $N_{2}$ intersecting $L_{1},L_{2}$. Consider a point $x$ on $N_{1}$ and the block $B_x$ determined by $x$ and $M$. Then $B_x$ intersects $N_2$. This implies that $L_{1}$, $L_{2}$ is a regular line pair. 
%
%Next, consider two skew lines $L_{1},L_{2}$ contained in a symp $Q$ and two blocks $B_{1}$ respectively $B_{2}$ containing them (this is possible by Corollary \ref{lineblock}). 
%Next consider a point $p_{1}$ on $B_{1}$ not in $Q$ and since $B_1\cap B_2=\emptyset$ it does not belong to $B_2$ either. Then $p_{1}$ is collinear with a point $p_{2}$ on $B_{2}$. If $p_{2}$ were contained in $Q$ then consider the projections $q_{1}$ (resp $q_{2}$) of $p_{1}$ (resp. $p_{2}$ onto $L_{1}$). The Quadrangle Lemma then implies that $p_{1}$ belongs to $Q$, a contradiction. So $p_{2}$ does not belong to $Q$. But this means we are exactly in the situation described in the above paragraph, proving (i). Also (ii) follows from the above proof.
\end{proof}

%An {\em O'Nan configuration} is a configuration of four distinct lines meeting in six distinct points.

We now arrive at the crux of this section. 

\begin{theorem}\label{nonclosing}
If the symplecta of $\Gamma$ are thick generalized quadrangles, then every maximal singular subspace contains a non-closing O'Nan configuration. In particular, no maximal singular subspace is a projective space.
\end{theorem}

%\begin{figure}
%\includegraphics[scale=0.30]{onan.pdf}
%\vspace{-3.4cm}
%\caption{Proof of Lemma~\ref{nonclosing}\label{pic}}
%\end{figure}

\begin{proof}
Let $H$ be any symp and let $\cS$ be a spread induced by some block of $\Delta$ outside $H$. By the previous lemma, it suffices to show that the geometry $\sigma$ with point set $\cS$ and lines the double perps contains non-closing O'Nan configurations. 

Consider lines $X,L_{1},L_{2},M_{1}\in\cS$ such that $X\in(\{L_{1},L_{2}\}^\perp)^\perp=:R$, $M_1\notin R$, see Figure~\ref{pic}. Consider a line $N\in R^{\perp}$ and the points $p_{1}=N\cap L_{1}$ and $p_{2}=N\cap L_{2}$. Let $q_{1}$ be the unique point of $M_1$ collinear to $p_{1}$. Consider the unique line $K$ incident with $q_1$ and concurrent with $X$.  
Let $q_{2}$ be the unique point of $K$ collinear to $p_{2}$.
Let $M_{2}\in(\{X,M_1\}^\perp)^\perp$ be such that $q_2\in M_2$. Suppose first there exists $Y\in(\{L_1,M_1\}^\perp)^\perp\cap(\{L_2,M_2\}^\perp)^\perp\cap\cL$. Then $Y\in S$ by Lemma~\ref{pairregular}$(ii)$.

As $p_iq_i\in\{L_i,M_i\}^\perp$, $i=1,2$, the line $Y$ meets both $p_1q_1$ and $p_2q_2$. Hence $Y,N,K\in\{p_1q_1,p_2q_2\}^\perp$. Since $X$ intersects both $N$ and $K$, we have $X\in(\{p_1q_1,p_2q_2\}^\perp)^\perp$. By regularity $Y$ is concurrent with $X$, a contradiction to $X,Y\in\cS$. Hence $(\{L_1,M_1\}^\perp)^\perp\cap(\{L_2,M_2\}^\perp)^\perp=\emptyset$, and the four double perps $(\{L_1,L_2\}^\perp)^\perp$, $(\{L_1,M_1\}^\perp)^\perp$, $(\{M_1,M_2\}^\perp)^\perp$ and $(\{L_2,M_2\}^\perp)^\perp$ form a non-closing O'Nan configuration. 
%Consider the opposite reguli $R_{i},i=1,2$ determined by the pairs $(L_{i},M_{i})$. Suppose $R_{1}$ and $R_{2}$ have an element $Y$ in common. Then $Y$ intersects $\langle p_{1},q_{1} \rangle$ as it belongs to $R_{1}$ and $\langle p_{2},q_{2} \rangle$ as it belongs to $R_{2}$. But then $Y$ belongs to the regulus determined by $K$ and $N$, whose opposite regulus contains $X$. Hence $X$ and $Y$ intersect. But $Y\in S$ as induced spreads are closed under taking reguli by (ii) of Lemma \ref{pairregular}, a contradiction. So $R_{1}$ and $R_{2}$ do not have an element in common, i.e. we have found a non-closing O'Nan configuration.
\end{proof}

Theorem~\ref{nonclosing} has been shown by Cohen in \cite{Cohen}, Proposition~4.2, for \emph{classical} generalized quadrangles, i.e., for generalized quadrangles occurring as point-residues in polar spaces of rank 3, with an extensive and explicit calculation. More recently,  Shult \& K. Thas~\cite{KoenErnie} have found a more elegant proof  only using the \emph{Moufang property} of such quadrangles (and in fact, they only use the \emph{strong transitivity}, i.e., the transitivity of the automorphism group on the set of pairs of incident point-line pairs $(p,L),(q,M)$ such that $p,q$ are not collinear and $L,M$ not concurrent). Finally, our proof above only  uses the fact that for some line $L$ of the spread, there exists at least one line $M$ not in the spread such that $\{L,M\}$ is a regular pair. 

Note that in the finite case one does not need any of the just mentioned additional conditions, as noted by both Cohen and Shult \& K. Thas in the above references. 

We now present a class of examples of imbrex geometries of symplecic rank 2 among which many with thick symps. In this case, the maximal singular subspaces can be seen as higher dimensional unitals.

Let $\Delta$ be any generalized quadrangle whose dual $\Omega$ is  embedded in some finite-dimensional projective space $\mathbb{P}^n(\L)$ over the skew field $\L$ which is finite-dimensional over its center, i.e., the blocks of $\Delta$ are points of $\mathbb{P}^n(\L)$ and the points of $\Delta$ are lines of  $\mathbb{P}^n(\L)$ (with natural incidence). All such $\Omega$ are classified, see \cite{Ste-Mal:99,Ste-Mal:00} (every classical generalized quadrangle in the above sense qualifies, except that for some the dimension of the projective space is infinite). Now let $\Gamma=(\cP,\cL,*)$ be the point-line geometry with same point set as $\Delta$, and where the lines correspond to planes of  $\mathbb{P}^n(\L)$ containing at least two lines of $\Omega$, with natural incidence. Axiom~(PPS1) is easy to check;  for Axiom~(PPS2) let $L,M$ be the lines of $\mathbb{P}^n(\L)$ corresponding to two non-collinear points of $\Gamma$. Then $\<L,M\>$ is 3-dimensional. Now $\<L,M\>$ induces a subquadrangle $H$ of $\Omega$, and $H$ is generated by $L,M$. Hence (PPS2) holds with the symplecta being the subquadrangles of $\Delta$ corresponding to $3$-spaces of  $\mathbb{P}^n(\L)$ containing two non-intersecting lines of $\Omega$, if we assume that these quadrangles are thick (this is needed for (PS1)). Finally, Axiom~(PPS4) follows from the finite-dimensonality of  $\mathbb{P}^n(\L)$ and of $\L$ over its center. 

We now check the Axiom (Imb). Translated to $\Omega$, we have to show that for given lines $L,M_{1},M_{2}$ of $\Omega$ such that $M_{1}$ and $M_{2}$ intersect and $L$ does not intersect the plane $\langle M_{1},M_{2}\rangle$, the 3-spaces $\langle L,M_{1} \rangle$ and $\langle L,M_{2} \rangle$ intersect in a plane $\pi$ containing at least two lines of $\Omega$. Clearly $\pi=\<L,M_1\cap M_2\>$ and this contains, besides $L$, also the line through $M_1\cap M_2$ intersecting $L$.  

In the finite case, there is just one class of examples, namely where $\Omega$ is the Hermitian quadrangle denoted by $Q(4,q^2)$. The maximal singular subspaces of $\Gamma$ are in this case classical unitals. In the infinite case, all quadrangles arising from a $\sigma$-quadratic form, with $\sigma$ nontrivial, qualify.

Similarly, other classes of examples are the exceptional Moufang quadrangles of types $\mathsf{F_4},\mathsf{E_6},\mathsf{E_7}$ and $\mathsf{E_8}$. Using the terminology of Appendix C of \cite{hvm}, if we let points correspond to the isotropic orbit in the Tits diagram which correspond to the multiple root in the root system of type $BC_2$, then the symplecta are the ideal subquadrangles belonging to a root system of type $C_2$. Again, all axioms of an imbrex geometry hold. 
%\subsection{Examples where there is a non-closing O'Nan configuration}

%Consider the dual Hermitian quadrangle $H(4,q^2)^D$, and let the symplecta be the subquadrangles of the order $(q,q^2)$. This is an example of an \HJ geometry. Indeed, two collinear points $x,y\in H(4,q^{2})^{D}$ correspond to two non-intersecting lines $L_{1},L_{2}$. The symplecton $[x,y]$ corresponds with the $H(3,q^{2})$ which is the intersection of $H(4,q^{2})$ with $\langle L_{1},L_{2}\rangle$, which is clearly convex. Expressing the Axiom for \HJ geometries in terms of $H(4,q^{2})$ yields that for given lines $L,M_{1},M_{2}$ such that $M_{1}$ and $M_{2}$ intersect and $L$ does not intersect the plane $\langle M_{1},M_{2}\rangle$ then the 3-spaces $\langle L,M_{1} \rangle$ and $\langle L,M_{2} \rangle$ intersect in a plane $\pi$ such that $\pi\cap H(4,q^{2})$ is a point over an $H(1,q^{2})$. This is obvious. Moreover collinearity of points in $H(4,q^{2})^{D}$ corresponds with concurrence of lines in $H(4,q^{2})$. Hence, the maximal singular subspaces of $H(4,q^{2})^{D}$ are unitals.

\section{Imbrex geometries of symplectic rank at least 3}\label{rank3}
In this section, $\Gamma=(\cP,\cL,*)$ is an imbrex geometry of symplectic rank $r$ at least 3.

\begin{lemma}\label{(CC1)}
Let $(x,H)$ be a non-incident point-symplecton pair in $\Gamma$. If $x$ is collinear with all points of a line $L\subseteq H$, then $x$ is collinear with all points of a maximal singular subspace of $H$.
\end{lemma}

\begin{proof}
Let $p\in H$ be a point which is not collinear to all points of $L$, and let $H'=\xi(p,x)$. Then by convexity $H'$ contains the unique point $p'$ on $L$ collinear to $p$ in $H$. Select $x'\in L$ different from $p'$. Then we claim that no point on the line $x x'$ is collinear to $p$. Indeed, suppose some point $y\in xx'$ is collinear to $p$, then $y\in H'$, and hence also $x'\in H'$, contradicting $H=\xi(p,x')$. Our claim is proved.

Now Axiom (Imb) implies that $H$ and $H'$ intersect in a maximal singular subspace $U$. In $H'$, the point $x$ is collinear with an $(r-1)$-subspace $U'$ of $U$, which does not contain $x'$ as  $x'\notin H'$. Hence $x$ is collinear with all points of $U'$ and with $x'$, and these must then generate a maximal singular subspace.
\end{proof}

Now we can apply Theorem 15.4.5 of \cite{Shu:12}. The latter theorem is an updated summary of the work done by Cohen \cite{Cohen} and Cohen \& Cooperstein \cite{COCO}. Since we assume diameter 2, our conclusion will only contain a restricted list of geometries. Without going into too much detail, we define the \emph{Lie incidence geometry} $\mathsf{X}_{n,i}(\K)$, where $\mathsf{X}\in\{\mathsf{A},\mathsf{D},\mathsf{E}\}$, as the $i$-Grassmannian point-line geometry related to the building of type $\mathsf{X}_n$ over the skew field $\K$ (which is automatically a field for the cases $\mathsf{D}$ and $\mathsf{E}$), and we use Bourbaki labeling \cite{Bourbaki} for the nodes of the corresponding diagram, and hence for the subscripts.  For instance, the Lie incidence geometry $\mathsf{A}_{n,i}$ is the ordinary Grassmannian geometry of all $i$-spaces of the $(n+1)$-dimensional vector space over $\K$. 

%\begin{theorem}\label{CoCo}
%Let $\Gamma$ be a strong parapolar space of diameter all of whose symplecta possess a constant finite symplectic rank $r\geq 3$. 
%Assume these hypotheses
%%\begin{itemize}\addtolength{\itemsep}{-0.4\baselineskip}
%\item $\Gamma$ is not itself a polar space
%\item If for any non-incident point-symplecton pair $(x,S)$, the intersection $x^{\perp}\cap S$ has projective rank at least %$r-1$, then in fact it has rank $r$.
%\end{itemize}
%Then one of the following conclusions must hold CONTROLEREN WELKE WE OVERHOUDEN
%\begin{itemize}\addtolength{\itemsep}{-0.4\baselineskip}
%\item[(r=3)] $\Gamma$ is either: 
%\begin{itemize}\addtolength{\itemsep}{-0.4\baselineskip}
%\item The Grassmannian of $k$-spaces of a (possibly infinite-dimensional) vector space $V$ over a division ring $D$
%\item The quotient $A_{2n-1,n}(D)/\langle \sigma \rangle$, where $\sigma$ is a polarity of $V$ of Witt index at most $n-5$.
%\end{itemize}
%\item[(r=4)] $\Gamma$ is a homomorphic image of a half-spin geometry of type $D_{n,n}$ over a field $\mathbb{K}$. 
%This homomorphism is an isomorphism if $n\leq 9$.
%\item[(r=5)] $\Gamma$ is the Lie incidence geometry $E_{6,1}(\mathbb{K})$
%\item[(r=6)] $\Gamma$ is the Lie incidence geometry $E_{7,7}(\mathbb{K})$ (Bourbaki numbering)
%\end{itemize}
%\end{theorem}

\begin{cor}\label{rank3classification}
An imbrex geometry $\Gamma$ of rank $r$ at least 3 is exactly one of the following.
\begin{itemize}\addtolength{\itemsep}{-0.4\baselineskip}
\item[$(r=3)$] The Lie incidence geometry $\mathsf{A_{n,2}}(\L)$, for $\L$ any skew field, $n\geq 4$. 
\item[$(r=4)$] The Lie incidence geometry $\mathsf{D_{5,5}}(\mathbb{K})$, for any field $\K$. 
\item[$(r=5)$] The Lie incidence geometry $\mathsf{E_{6,1}}(\mathbb{K})$, for any field $\K$.
\end{itemize}
\end{cor}

\begin{proof}
If $\Gamma$ is an imbrex geometry of rank at least $3$, then the assertion follows from Theorem 15.4.5 of \cite{Shu:12}. Conversely,  if $\Gamma$ is one of the mentioned geometries, then we will prove in the next section that their universal embedding is a local Mazzocca-Melone set (for definitions, see the next section), and this will immediately imply that they satisfy (Imb). Hence we postpone the details to the next section.
\end{proof}

\section{Local Mazzocca-Melone sets}\label{Application}
We now apply the classification of the previous section to the theory of Mazzocca-Melone sets, initiated in \cite{JSHVMBIG}. We introduce pre-Mazzocca-Melone sets and local Mazzocca-Melone sets in the next subsection. 

\subsection{Definition of local Mazzocca-Melone sets}
Let $N,d,r$ be natural numbers. Let $X$ be a spanning point set of  $\mathbb{P}^n(\K)$, with $\mathbb{K}$ any skew field, and let $\Xi$ be a collection of $(d+1)$-spaces of  $\mathbb{P}^n(\K)$, $d\geq 2$, such that, for any $\xi\in\Xi$, the intersection $\xi\cap X=:X(\xi)$ is a polar space of rank $r$, $r\geq 2$,
in $\xi$ (and then, for $x\in X(\xi)$, we denote the tangent space at $x$ to $X(\xi)$ by $T_x(X(\xi))$ or sometimes simply by $T_x(\xi)$). Also, we say that two points of $X$ are $X$-\emph{collinear} if all points of the line they span in  $\mathbb{P}^n(\K)$ are contained in $X$. We call $(X,\Xi)$ a \emph{pre-Mazzocca-Melone set (of type $(d,r)$)} if (MM1) and (MM2) are satisfied, and a \emph{local Mazzocca-Melone set} if on top (LMM3) holds. 

\begin{itemize}\addtolength{\itemsep}{-0.4\baselineskip}
\item[(MM1)] Any pair of points $x$ and $y$ of $X$ which are not $X$-collinear lies in at least one element of $\Xi$.

\item[(MM2)] If $\xi_1,\xi_2\in \Xi$, with $\xi_1\neq \xi_2$, then $\xi_1\cap\xi_2\subset X$.
\end{itemize}

It follows from (MM1) and (MM2) that any pair of points $x$ and $y$ of $X$ which are not $X$-collinear lies in exactly one member of $\Xi$, and we denote that member by $[x,y]$. The pre-Mazzocca-Melone set $(X,\Xi)$ is called \emph{proper} if $|\Xi|\geq 2$. It follows from these axioms that proper pre-Mazzocca-Melone sets define parapolar spaces where the symps are the polar spaces $X(\xi)$, $\xi\in\Xi$, see~\cite{JSHVMBIG}. 

\begin{itemize}\addtolength{\itemsep}{-0.4\baselineskip}
\item[(LMM3)] If $x\in X$ and $L\subseteq X$ is a line of  $\mathbb{P}^n(\K)$ such that no point of $L$ is $X$-collinear with $x$, then all $d$-spaces $T_x([x,y])$, $y\in L$, generate a subspace  $T_{x,L}$ of  $\mathbb{P}^n(\K)$ of dimension at most $2d-r+1$.
\end{itemize}

A local Mazzocca-Melone set is \emph{proper} if it is a proper pre-Mazzocca-Melone set. Non-proper pre-Mazzocca-Melone sets are just embedded polar spaces. In this section we shall classify all proper local Mazzocca-Melone sets of type $(d,r)$, for all $d,r\geq 2$.

\subsection{Examples of proper local Mazzocca-Melone sets}
In \cite{Tha-Mal:13}, it is proved that the \emph{Segre variety} $\mathcal{S}_{p,q}(\K)$, $p\geq 1$ and $q\geq 1$, satisfies the axioms of a local Mazzocca-Melone set (although it was not called as such in \cite{Tha-Mal:13}). 

Now, the Lie incidence geometries $\mathsf{A_{n,2}}(\L)$, for $\L$ any skew field, $n\geq 4$, $\mathsf{D_{5,5}}(\mathbb{K})$, for any field $\K$, and $\mathsf{E_{6,1}}(\mathbb{K})$, for any field $\K$, all admit the so-called \emph{universal embedding}, which is a pre-Mazzocca-Melone set, as proved in \cite{JSHVMBIG}. We denote these pre-Mazzocca-Melone sets by $\mathcal{A}_{p,2}(\L)$, $p\geq 4$, $\mathcal{D}_{5,5}(\K)$ and $\mathcal{E}_{6,1}(\K)$, respectively.% satisfy the above axioms. We use slightly different notation here since we view them as varieties.

We now show that these are local Mazzocca-Melone sets. By inclusion of the appropriate sets, it suffices to show (LMM3) for the minimal cases, namely, $\mathcal{A}_{4,2}(\L)$, $\mathcal{D}_{5,5}(\K)$ and $\mathcal{E}_{6,1}(\K)$. Note that in this list, a preceding one is the \emph{residue} of the next one, i.e., if $x$ is a point in one of these sets $X$, then the lines of $X$ through $x$ form, together with the tangent spaces to the symps at $x$, a pre-Mazzocca-Melone set in a subspace of the quotient projective space with respect to $x$ isomorphic to the previous Mazzocca-Melone set in the list. Hence the assertion will be proved if we show that the validity of (LMM3) for a given set follows from the validity of that axiom in any residue, adding $\mathcal{S}_{1,2}(\K)$ in front of that list. %The latter is quite trivial, since the whole tangent space $T_p$ at a point is just $3$-dimensional (spanned by a singular line and a singular plane).

So suppose (LMM3) is valid in the residue $(X_p,\Xi_p)$ of $(X,\Xi)$ at $p\in X$, for all $p\in X$, with $(X,\Xi)$ one of the pre-Mazzocca-Melone sets $\mathcal{A}_{4,2}(\L)$, $\mathsf{D}_{5}(\K)$ or $\mathsf{E}_{6,1}(\K)$. Let $r-1$ be the dimension of the maximal singular subspaces of the symps (having rank $r$) of $(X,\Xi)$ (which are then hyperbolic quadrics in $(2r-1)$-dimensional spaces, and $d=2r-2$). Let $x\in X$ and let $L$ be a line contained in $X$ no point of which is $X$-collinear with $x$. Let $H$ and $H'$ be two distinct symps containing $x$ and a (different) point of $L$, say $y,y'$, respectively. We first claim that $H\cap H'$ has dimension $r-1$ (hence that (Imb) holds). Since $y'$ is collinear with the point $y$ of $H$, it is collinear with an $(r-1)$-space $U$ of $H$ (this can be checked in all instances directly from the definitions of the corresponding Lie incidence geometries). There is a unique $(r-2)$-space $W\subseteq U$ all of whose points are collinear with $x$. By the Quadrangle Lemma, the singular $(r-1)$-space $\<x,W\>$ is contained in $H'$, showing our claim. 

Moreover, since both $y$ and $y'$ are collinear with all points of $W$, all points of $L$ are collinear with all points of $W$ and $\<L,W\>$ is a singular $r$-space. The Quadrangle Lemma implies that any symp $H^*$ through $x$ and a point of $L$ contains $W$. Since $r\geq 3$, we can select two distinct points $q,q'$ of $W$.  Clearly, $\<H^*\>$ is generated by $T_q(H^*)$ and $T_{q'}(H^*)$. Hence $\<T_{x,L},L\>$ is generated by all $T_q(H^*)$ and $T_{q'}(H^*)$, for $H^*$ running through all symps through $x$ and a point of $L$. By induction, the dimension of the span of all $T_q(H^*)$ for $H^*$ as above, is equal to $2(2r-4)-(r-1)+1+(2+1)=3r-3$ (in the residue at $q$, the point $\<q,x\>$ is not $X_q$-collinear to any point of $\<q,L\>$). Similarly for the span of all $T_{q'}(H^*)$. The intersection of those two spaces is, by (double) induction $2(2r-6)-(r-2)+1+(2+2)=3r-5$ (look in the residue of the line $\<q,q'\>$; for $r=2$, there is no ``double'' induction, but then this can be seen directly in $\mathcal{S}_{1,2}(\K)$). Hence the dimension of   $\<T_{x,L},L\>$ is equal to $3r-1$. Hence, since $L$ does not meet $T_x(H)$, we see that $\dim(T_{x,L})=3r-3=2(2r-2)-r+1$, exactly what we had to prove.

%{\bf Remarks} 
%Note that for $d=1$ condition (LMM3) is void. In that case one can still something in case one asks an extra dimension %condition, namely in \cite{JSHVMBIG} we proved

%\begin{lemma}\label{not3}
% A pre-Mazzocca-Melone set of split type $1$ in $\PAG^5(\K)$, and $|\K|>2$ is equivalent to a quadric Veronesean variety %$\mathcal{V}_2(\K)$.
%\end{lemma}

\subsection{Classification of Local Mazzocca-Melone sets}\label{sec:main results}

Using the results on \HJ geometries we now classify all local Mazzocca-Melone sets, up to projection from a suitable subspace (i.e., a subspace containing no point in the span of two symplecta). % description of the varieties below can be found for instance in \cite{JSHVMBIG}.

%VOORBEELD VAN EEN DERGELIJKE PROJECTIE INVOEGEN?

\begin{theorem}\label{classification}
{A proper local Mazzocca-Melone set of type $(d,r)$, $d,r\geq 2$ is projectively equivalent to a projection from a suitable subspace (in the above sense) of one of the following pre-Mazzocca-Melone sets.
\begin{itemize}\addtolength{\itemsep}{-0.4\baselineskip}
\item[$d=2$:]
the Segre variety $\mathcal{S}_{p,q}(\K)$, $p\geq 1$ and $q\geq 1$, %(non-proper when $p=q=1$);
\item[$d=4$:]
$\mathcal{A}_{p,2}(\K)$, $p\geq 4$, %(non-proper when $p=3$);
\item[$d=6$:]
$\mathcal{D}_{5,5}(\K)$, %(non-proper when $p=3$);
\item[$d=8$:] $\mathcal{E}_{6,1}(\K)$.
%\item[or:] a non-proper variety corresponding to a split quadric in $\PAG^{d+1}(\K)$, for $d\notin\{2,4,6,8\}$.%the $\mathsf{D}_{\frac{d}{2}+1,1}(\K)$ variety (non-proper).
\end{itemize}
%In all cases, $\K$ is commutative. There do not exist proper local Mazzocca-Melone sets of hyperbolic type $n\geq 5$.
}
\end{theorem}
%\bigskip

\begin{proof} 
Clearly, a local Mazzocca-Melone set is an \HJ geometry. %The only thing which requires some work is (PPS3), see Remark 4.3 of \cite{JSHVMBIG}. 
So we can apply the results of Section \ref{Imbrex}.
Suppose first that $r=2$. Then the symplectic rank as a parapolar space is 2, and as the geometry is embedded, all singular subspaces are projective spaces and hence do not contain non-closed O'Nan configurations. By Theorem \ref{nonclosing}, we are done. If $r\geq 3$, we use Corollary \ref{rank3classification}, and the fact that all embeddings arise from the universal embedding by suitable projection.
\end{proof}

\subsection{Weakening the (Imb) axiom}

If we would not care about the residual property of (Imb), then an interesting option would be to weaken it to the following condition.
\begin{itemize}
\item[(Imb$^*$)]  \em Let $x$ be a point not collinear with any point of the line $L$. Let $y_1,y_2$ be distinct points on $L$. Then the symplecta $\xi(x,y_1)$ and $\xi(x,y_2)$ intersect in a singular subspace of dimension at least $1$.   
\end{itemize}
This does not guarantee constant symplectic rank from the beginning, and we were not able to classify strong parapolar spaces of rank 2 and symplectic rank at least $3$ under Condition~(Imb$^*$). However, it would be interesting to do so, since, from the point of view of Hjelmslev, this is the weakest condition one can ask. We pose it here as an open problem.

{\bf Acknowledgement} The authors would like to thank the Mathematisches Forschungsinstitut Oberwolfach for providing them with the best possible research environment imaginable during a stay as Oberwolfach Leibniz Fellow (JS) and Visiting Senior Researcher (HVM) in June 2013.

\end{document}